\documentclass[12pt, reqno]{amsart}
\usepackage{amsmath, amsthm, amscd, amsfonts, amssymb, graphicx, color}
\usepackage{amssymb, enumerate}

\newtheorem{theorem}{Theorem}[section]
\newtheorem{lemma}[theorem]{Lemma}
\newtheorem{proposition}[theorem]{Proposition}

\theoremstyle{definition}
\newtheorem{definition}[theorem]{Definition}

\theoremstyle{remark}

\numberwithin{equation}{section}

\newcommand{\real}{{\mathbb R}}

\newcommand{\A}{{\mathcal A}}

\newcommand{\E}{{\mathcal E}}
\newcommand{\F}{{\mathcal F}}

\newcommand{\M}{{\mathcal M}}
\newcommand{\N}{{\mathcal N}}
\newcommand{\R}{{\mathcal R}}

\newcommand{\8}{\infty}
\newcommand{\el}{\ell}

\newcommand{\be}{\begin{eqnarray*}}
\newcommand{\ee}{\end{eqnarray*}}
\newcommand{\beq}{\begin{equation}}
\newcommand{\eeq}{\end{equation}}
\newcommand{\beqn}{\begin{equation*}}
\newcommand{\eeqn}{\end{equation*}}
\newcommand{\bsp}{\begin{split}}
\newcommand{\esp}{\end{split}}

\begin{document}

\setcounter{page}{1}

\title[On Haagerup noncommutative $H^p(\A)$ spaces]{On Haagerup noncommutative quasi $H^p(\A)$ spaces}

\author[T. N. Bekjan ]{Turdebek N. Bekjan}

\address{Astana IT University,  Astana 010000, Kazakhstan}
\email{ bekjant@yahoo.com}



\begin{abstract}
Let $\mathcal{M}$ be a $\sigma$-finite von Neumann algebra, equipped with a
 normal  faithful state $\varphi$,  and  let $\mathcal{A}$ be a maximal subdiagonal subalgebra of $\mathcal{M}$.
 We have  proved that  for $0< p<1$, $H^p(\mathcal{A})$ is independent of $\varphi$. Furthermore, in the case that $\mathcal{A}$ is a type 1 subdiagonal subalgebra, we have extended the most recent results about the Riesz type factorization to the case $0<p<1$ and have proved an interpolation theorem for $H^p(\mathcal{A})$ in the case where $0 < p_0, p_1 \le \infty$.
\end{abstract} \maketitle

\section{Introduction}
Let $(\Omega, \mu)$ be a measure space. It is well-known that if $0<\theta<1$
and  $0< p_0,\;
p_1\le\8$, then
 \beq\label{eq:interpolation-classical}
 \big(L^{p_0}(\Omega),\; L^{p_1}(\Omega)\big)_\theta=L^p(\Omega)
\eeq
with equal norms, where $1/p=(1-\theta)/p_0 + \theta/p_1$ (see \cite{Ca,CaT}). If $\mathcal{M}$ is a semifinite von Neumann algebra equipped with a normal
semifinite faithful trace $\tau$ and  $L^p(\mathcal{M},\tau)$  is the tracial noncommutative
$L^p$-space associated with $(\mathcal{M},\tau)$, then the noncommutative version of \eqref{eq:interpolation-classical} holds, i.e.,
for any $0< p_0,\;
p_1\le\8$ and $0<\theta<1$,
\beq\label{eq:interpolation-semifinite}
 \big(L^{p_0}(\mathcal{M},\tau),\; L^{p_1}(\mathcal{M},\tau)\big)_\theta=L^p(\mathcal{M},\tau)
\eeq
with equivalents norms, where $1/p=(1-\theta)/p_0 + \theta/p_1$ (see \cite{PX,X}.

Let  $\mathcal{M}$  be a $\sigma$-finite von Neumann algebra equipped with a faithful
 normal   state $\varphi$ and let $L^p(\mathcal{M}, \varphi)$
be the Haagerup  noncommutative $L^p$-spaces associated with $(\mathcal{M}, \varphi)$.  Kosaki \cite{Kos} proved  the Haagerup noncommutative
$L^{p}$-space analogue of the \eqref{eq:interpolation-semifinite} for the case that $1\le p_0,\;
p_1\le\8$. In \cite{GYZ}, this result was extended for the case $0< p_0,\;
p_1\le\8$.

Ji \cite{J1} studied  Haagerup noncommutative
$H^{p}$-spaces $H^p(\mathcal{A})$ based on Haagerup noncommutative
$L^p$-space $L^p(\mathcal{M})$ associated with a $\sigma$-finite von Neumann algebra $\mathcal{M}$ and a maximal subdiagonal
    algebra $\mathcal{A}$ of $\mathcal{M}$.  He proved that for $1\le p<\8$,  $H^p(\mathcal{A})$ is independent of $\varphi$ (see  \cite[Theorem 2.5]{J1}). We will extend this result for the case $0< p<1$.
In \cite{BR1}, the Haagerup noncommutative
$H^{p}$-space version  of  \eqref{eq:interpolation-semifinite} for the case $1\le p_0,\;
p_1\le\8$  has been proved. In this paper,  we  generalize this result to  the case  where $0< p_0,\;
p_1\le\8$ and $\mathcal{A}$ is a type 1 subdiagonal algebra (see Section 4).

The organization of the paper is as follows: In Section 2, we
give some definitions and prove that $H^p(\mathcal{A})$ is independent from  $\varphi$  for $0< p<1$.   Section 3 is devoted to discussing complex interpolation for a couple of
(quasi)   Haagerup
noncommutative $H^{p}$-spaces. Section 4 presents a complex interpolation theorem for the case where $0< p_0,\;
p_1\le\8$ and  $\mathcal{A}$ is a type 1 subdiagonal algebra.

\section{Preliminaries}

In the context of this paper, let $\mathcal{M}$ be a $\sigma$-finite von Neumann algebra on a complex Hilbert space $\mathcal{H}$, equipped with a faithful and normal state denoted as $\varphi$. We shall briefly recall the definition of Haagerup noncommutative $L^p$-spaces associated with $\mathcal{M}$.

Let $\{\sigma_{t}^\varphi\}_{t\in\mathbb{R}}$  be the one-parameter modular automorphism group of $\mathcal{M}$ associated with $\varphi$ and denote the crossed product of $\mathcal{M}$ by this group as $\mathcal{N} =\mathcal{M}\rtimes_{\sigma^\varphi}\mathbb{R}$.  We consider the canonical embedding $\pi(M)$  of $\mathcal{M}$ into $\mathcal{N}$ and identify $\mathcal{M}$ with $\pi(\mathcal{M})$ within $\mathcal{N}$. There is  a  dual action $\{\hat{\sigma}_{t}\}_{t\in\mathbb{R}}$ of  $\mathbb{R}$ on $\mathcal{N}$. Then, we can express $\mathcal{M}$ as the set of elements in $\mathcal{N}$ that remain invariant under this dual action:
$$
\mathcal{M}=\{x\in\mathcal{N}: \; \hat{\sigma}_{t}(x)=x,\;\forall
t\in\mathbb{R}\}.
$$
Recall that $\N$ is  semi-finite (cf. \cite{PT}), and therefore, there exists a normal, semi-finite, and faithful trace denoted as $\tau$ defined on $\mathcal{N}$ that satisfies the equation
$$
\tau\circ\hat{\sigma}_{t}=e^{-t}\tau,\quad \forall t\in\mathbb{R}.
$$
Now, by denoting $L_0(\mathcal{N})$ as the set of $\tau$-measurable operators affiliated with $\mathcal{N}$, we can define the Haagerup noncommutative $L^p$-spaces, where $0<p\leq\infty$, as follows
$$
L^{p}(\mathcal{M},\varphi)=\{x\in L_{0}(\N,\tau):
\;\hat{\sigma}_{t}(x)=e^{-\frac{t}{p}}x,\;\forall t\in\mathbb{R}\}.
$$
For the space $L^1(\M)$, it possesses a canonical trace functional denoted as $tr$, which is used to define a norm $\|x\|_p = tr(|x|^p)^\frac{1}{p}$ on $L^p(\M)$. The topology on $L^p(\M)$ induced by this norm coincides with the relative topology of convergence in measure inherited by $L^p(\M)$ from $L_0(\N)$. For $0 < p <\infty$ and any element $x$ in $L^p(\mathcal{M},\varphi)$, we have the relationship
\beq\label{eq:generalized singular-haagerup}
\mu_t(x)=t^{-\frac{1}{p}}\|x\|_p,\qquad t>0,
\eeq
where $\mu_t(\cdot)$  is relative to $(\N,\tau)$ (see \cite[Lemma 4.8]{FK}).
Additionally, it is worth noting that
$$
L^{\infty}(\M,\varphi)=\M.
$$
Finally, according to \cite[Theorem 3.1]{L3}, for any element $x\in \mathcal{M}$, we have that
$$
\|x\|_\8 = \mu_t(x)\qquad \mbox{for all}\; 0< t \le 1.
$$
Recall that the dual weight $\hat{\varphi}$ of $\varphi$ has a Radon-Nikodym derivative $D$ with respect to $\tau$ (cf. \cite{PT}). This $D$ is an invertible positive self-adjoint operator affiliated with $\mathcal{N}$ and satisfies the following relationships
$$
\hat{\varphi}(x)=\tau(Dx),\quad x\in \N_{+},
$$

$$
\varphi(x)=tr(Dx),\quad \forall x\in \mathcal{M}
$$
and the modular automorphism group $\{\sigma_{t}^\varphi\}_{t\in\mathbb{R}}$ is given
by
$$
\sigma_{t}^\varphi(x)=D^{it}xD^{-it}, \quad x\in\mathcal{M},\; t\in\mathbb{R}.
$$

It is a well-known fact that $L^{p}(\mathcal{M},\varphi)$ is, up to isometry, independent of the choice of $\varphi$. Therefore, we will use the notation $L^{p}(\mathcal{M})$ to denote the abstract Haagerup noncommutative $L^{p}$-space, i.e., $L^{p}(\mathcal{M},\varphi)$.

For $0<p\leq\infty$ and a subset $K \subset L^{p}(\mathcal{M})$, we denote by $[K]_{p}$ the closed linear span of $K$ within $L^{p}(\mathcal{M})$ (relative to the w*-topology for $p =\infty$). For $0<p<\infty$ and $0\leq \eta \leq 1$, it can be obtained that
$$
L^{p}(\mathcal{M})=[D^{\frac{1-\eta}{p}}\M D^{\frac{\eta}{p}}]_p
$$
( refer to \cite[Lemma 1.1]{JX} for details).

Next, consider the interval $(0,\infty)$ equipped with the standard Lebesgue measure $m$. We denote $L_0(0,\infty)$ as the space of real-valued functions $f$ on $(0,\infty)$ that are measurable with respect to $m$ and satisfy $m(\{\omega \in (0,\8) :\; |x(\omega)| > s\})<\8$ for some $s$. The decreasing rearrangement function $f^*: [0, \infty) \rightarrow [0, \infty]$ for $f \in L_0 (0,\infty)$ is defined as follows
$$
f^*(t) = \inf\{s > 0 : \; m ( \{\omega \in (0,\8) :\; |f (\omega)| > s\}) \le t\}
$$
for $t \geq 0$.

The classical weak $L_p$-space, denoted as $L_{p,\infty}(0,\infty)$, is defined for $0<p<\infty$ as the set of all measurable functions $f$ on $(0,\infty)$ such that
\be
\|f\|_{Lp,\8}=\sup_{t>0}t^\frac{1}{p}f^*(t)<\8.
\ee

Let  $S^+(0,\8)=\{\chi_E\in(0,\8)\;:\; m(E)<\8\}$ and $S(0,\8)$ be the linear span of $S^+(0,\infty)$. We then define $L^{p,\8}_0(0,\8)$ as the closure of  $S(0,\8)$  within $L^{p,\8}(0,\8)$.

For $0< p < \8$, the noncommutative weak $L_p$ space $L^{p,\8}(\N)$ consists of operators $x$ in $L_{0}(\N)$ such that
\be
\|x\|_{p,\8}= \sup_{t > 0} t^{\frac{1}{p}}\mu_{t}(x) < \8.
\ee
When equipped with the norm $\|.\|_{p,\8}$, $L^{p,\8}(\N)$ forms a quasi Banach space. However, for $p>1$, it is possible to renorm $L^{p,\infty}(\N)$ as a Banach space by
\be
x \mapsto \sup_{t >0} t^{-1+ \frac{1}{p}} \int_{0}^{t}\mu_{s}(x)d s.
\ee
Alternatively, the quasi-norm can be described as follows
\beq\label{eq:WeakLpNormMu}
\|x\|_{p,\8} = \inf \big \{c>0:\; t ( \mu_t (x)/c )^p \le 1,\; \forall t >0 \big \}.
\eeq
Moreover, another useful representation in terms of the distribution function is given by
\beq\label{eq:WeakLpNormLmbda}
\|x\|_{p,\8} = \sup_{s > 0} s \lambda_s (x)^{\frac{1}{p}}.
\eeq

Recall that noncommutative weak $L_p$-spaces can be represented through noncommutative Lorentz spaces. Refer to  \cite{DDP2} and \cite{X1} for details.

Let us define $S^+(\N)$ as the set of elements $x$ in $\N^+$ such that $\tau(s(x))<\infty$, and let $S(\N)$ be the linear span of $S^+(\N)$. The closure of $S(\N)$ in $L^{p,\infty}(\N)$ is denoted as $L^{p,\infty}_0(\N)$. It is worth noting that
$$
(L^{p,\infty}_0(0,\infty))(\N)=L^{p,\infty}_0(\N)
$$
 (for definition of noncommutative symmetric space  see \cite{S}).

For any $0 < p < \infty$ and any element $x$ in $L^p(\M)$, using  \eqref{eq:generalized singular-haagerup}, we obtain that
\be
\|x\|_p=\sup_{t > 0} t^{\frac{1}{p}}\mu_{t}(x)=\|x\|_{p,\infty}.
\ee
Consequently, we get the following result.

\begin{theorem}\label{thm:Lp-isometric weak Lp}
For $0 < p \leq \infty$, $L^{p}(\M)$ is isometrically isomorphic to a subspace of the noncommutative weak $L^p$ space $L^{p,\infty}_0(\N)$.
\end{theorem}

Let us consider $\mathcal{D}$ as a von Neumann subalgebra of $\M$, and let $\mathcal{E}$ be a normal, faithful conditional expectation from $\M$ to $\mathcal{D}$.

Now, we introduce the following definition.

\begin{definition} Suppose $\mathcal{A}$ is  a w*-closed subalgebra  of $\mathcal{M}$. We refer to $\mathcal{A}$ as a subdiagonal subalgebra of $\mathcal{M}$ with respect to
$\mathcal{E}$ (or to $\mathcal{D}$)  if the following conditions are hold.
\begin{enumerate}[\rm (i)]

\item $\mathcal{A}+ J(\mathcal{A})$ is w*-dense in  $\mathcal{M}$,

\item $\mathcal{E}(xy)=\mathcal{E}(x)\mathcal{E}(y),\; \forall\;x,y\in
\mathcal{A}$,

\item $\mathcal{A}\cap J(\mathcal{A})=\mathcal{D}$,
\end{enumerate}
We also refer to $\mathcal{D}$ as the diagonal of $\mathcal{A}$.
\end{definition}

It's noteworthy that Arveson \cite{A} does not assume subdiagonal subalgebras to be w*-weakly closed. However, since the weak* closure of a  subalgebra that is subdiagonal subalgebra with respect to $\mathcal{E}$ will also be subdiagonal subalgebra with respect to $\mathcal{E}$ (as pointed out in Remark 2.1.2 in \cite{A}), we may consistently assume that our subdiagonal subalgebras are always w*-weakly closed, as defined in \cite{J1, J2, X}. Since $\M$ is $\sigma$-finite, we
can take a faithful normal state $\phi$  on $\M$ such that $\phi\circ\E=\phi$.  It is well-known  that the existence of a (unique)
normal conditional expectation $\E :\M \rightarrow \mathcal{D}$ such that  $\varphi\circ\E=\varphi$  is equivalent to $\sigma_{t}^\varphi(\mathcal{D})=\mathcal{D}$  for all $t\in\mathbb{R}$ (cf. \cite{T3}).
Henceforth, in the  rest  of this paper, $\varphi$ consistently denotes a normal faithful state that satisfies $\varphi\circ\E=\varphi$.

The  subalgebra $\mathcal{A}$ is considered a maximal subdiagonal  subalgebra of $\mathcal{M}$ with
respect to $\mathcal{E}$ (or to $\mathcal{D}$), if $\mathcal{A}$ is not properly contained
in any other subalgebra of $\mathcal{M}$ that  is itself a subdiagonal subalgebra with respect to $\mathcal{E}$.
Additionally, we define
$$
\mathcal{A}_{0}=\{x\in \mathcal{A}:\;\mathcal{E}(x)=0\}.
$$
According to \cite[Theorem 2.2.1]{A}, $\mathcal{A}$ is  maximal if and only if it satisfies the condition
$$
 \mathcal{A}=\{x\in
\mathcal{M}:\;\mathcal{E}(yxz)=0,\;\forall y\in
\mathcal{A},\;\forall z\in \mathcal{A}_{0}\}.
$$

Using \cite[Theorem 2.4]{JOS} and \cite[Theorem 1.1]{X}, we deduce that a subdiagonal subalgebra $\A$ of $\mathcal{M}$ with respect to $\mathcal{D}$ is maximal if and only if it satisfies the condition
\begin{equation}\label{maximal}
\sigma_{t}^\varphi(\A)=\A,\qquad \forall t\in\mathbb{R}
\end{equation}
 (refer to \cite[Theorem 1.1]{L2}).

Throughout this paper, we will consistently denote $\mathcal{A}$ as a maximal subdiagonal subalgebra within $\mathcal{M}$ with respect to $\mathcal{E}$.

\begin{definition}\label{def:hp}   For $0<p<\infty$,  we define the Haagerup
noncommutative $H^{p}$-space that
\be
H^{p}(\mathcal{A})=[\mathcal{A}D^{\frac{1}{p}}]_{p},\quad H^{p}_{0}(\mathcal{A})=
 [\mathcal{A}_{0}D^{\frac{1}{p}}]_{p}.
\ee
\end{definition}

 Let $1\le p<\infty,\; 0\le \eta\le1$. By \cite[Proposition 2.1]{J2}, we know that
\beq\label{eq:equalityHp}
H^{p}(\mathcal{A})=[D^{\frac{1-\eta}{p}}\A D^{\frac{\eta}{p}}]_p,\quad H^{p}_0(\mathcal{A})=[D^{\frac{1-\eta}{p}}\A_0 D^{\frac{\eta}{p}}]_p.
\eeq

It is known that
\beq\label{eq:Lp(D)-property}
L^{p}(\mathcal{D})=[D^{\frac{1-\eta}{p}}\mathcal{D} D^{\frac{\eta}{p}}]_p,\quad  \forall p\in (0,\infty),\; \forall \eta\in [1,0].
\eeq

\begin{proposition}\label{lem:hp} Let $0< p<1,\; 0\le \eta\le1$. Then
\be
H^{p}(\mathcal{A})=[D^{\frac{1-\eta}{p}}\A D^{\frac{\eta}{p}}]_p,\quad H^{p}_0(\mathcal{A})=[D^{\frac{1-\eta}{p}}\A_0 D^{\frac{\eta}{p}}]_p.
\ee
\end{proposition}
\begin{proof}
First assume that $\frac{1}{2}\le p<1$. If $0<\eta\le\frac{1}{2}$, by  \eqref{eq:equalityHp}, we have that
\be
[\A D^{\frac{\eta}{p}}]_{\frac{p}{\eta}}=[D^{\frac{\eta}{p}}\A ]_{\frac{p}{\eta}}.
\ee
Hence,
\be
H^{p}(\A)&=&[\A D^{\frac{1}{p}}]_{p}=[[\A D^{\frac{1}{2p}}]_{2p}D^{\frac{1}{2p}}]_p\\
&=&[[D^{\frac{1}{2p}}\A ]_{2p}D^{\frac{1}{2p}}]_p=[D^{\frac{1}{2p}}\A D^{\frac{1}{2p}}]_{p}\\
&=&[D^{\frac{1}{2p}}[ D^{\frac{1}{2p}}\A]_{2p}]_{p}=[D^{\frac{1}{p}}\A ]_{p}\\
&=&[D^{\frac{1-\eta}{p}}[D^{\frac{\eta}{p}}\A]_{\frac{p}{\eta}}]_p=[D^{\frac{1-\eta}{p}}[\A D^{\frac{\eta}{p}}]_{\frac{p}{\eta}}]_p\\
&=&[D^{\frac{1-\eta}{p}}\A D^{\frac{\eta}{p}}]_p.
\ee
If $\frac{1}{2}<\eta<1$, then $0<1-\eta<\frac{1}{2}$. Similarly,
\be
H^{p}(\A)&=&[\A D^{\frac{1}{p}}]_{p}=[[\A D^{\frac{1-\eta}{2p}}]_{\frac{p}{1-\eta}}D^{\frac{\eta}{p}}]_p\\
&=&[[D^{\frac{1-\eta}{p}}\A ]_{\frac{p}{1-\eta}}D^{\frac{\eta}{p}}]_p=[D^{\frac{1-\eta}{p}}\A D^{\frac{\eta}{p}}]_p.
\ee
In the case $\frac{1}{4}\le p<\frac{1}{2}$, we use the result of the first case  and similar method to get that
$$
H^{p}(\mathcal{A})=[D^{\frac{1-\eta}{p}}\A D^{\frac{\eta}{p}}]_p,\qquad \mbox{for each}\; 0<\eta\le1. 
$$ Iterating
this procedure, we obtain the desired result.

The second result follows analogously.
\end{proof}

For $1\le p<\8$,  $H^p(\A)$ is independent of $\varphi$ (see  \cite[Theorem 2.5]{J1}). Next, we extend this result for the case $0< p<1$.
Let $\psi$ be a
faithful normal state such that $\psi\circ\E=\psi$. Denote by $D_{\psi}$ the Radon-Nikodym derivative
of  $\psi$'s the dual
weight $\hat{\psi}$   with respect to
$\tau$. Then for $x\in\mathcal{N}_{+}$ and $0<p<\8$,
\beq\label{eq:psi}
\hat{\psi}(x)=\tau(D_{\psi}x)\qquad \mbox{and}\qquad D_{\psi}^\frac{1}{p}\in L^{p}(\mathcal{D})_{+}.
\eeq
We have the following result.
\begin{theorem}\label{pro:hp-idependent} Let $0< p<1$ and $\psi$ be as in the above. Then
\be
[\mathcal{A}D_\psi^{\frac{1}{p}}]_{p}=[\mathcal{A}D^{\frac{1}{p}}]_{p},\qquad [\mathcal{A}_{0}D_\psi^{\frac{1}{p}}]_{p}=
 [\mathcal{A}_{0}D^{\frac{1}{p}}]_{p}
\ee
and
\be
[\mathcal{D}D_\psi^{\frac{1}{p}}]_{p}=
 [\mathcal{D}D^{\frac{1}{p}}]_{p}.
\ee
\end{theorem}
\begin{proof}
Since $[\mathcal{D}D^{\frac{1}{p}}]_{p}=L^{p}(\mathcal{D})$, by \eqref{eq:psi}, we get that
$$
[\mathcal{D}D_\psi^{\frac{1}{p}}]_{p}\subset
 [\mathcal{D}D^{\frac{1}{p}}]_{p}.
 $$
 On the other hand, we have that
 $$
 [\mathcal{D}D_\psi^{\frac{1}{p}}]_{p}=L^{p}(\mathcal{D})\quad\mbox{and}\quad D^\frac{1}{p}\in L^{p}(\mathcal{D})_{+}.
 $$
  Hence, we obtain that
  $$
  [\mathcal{D}D^{\frac{1}{p}}]_{p}\subset
 [\mathcal{D}D_\psi^{\frac{1}{p}}]_{p},
 $$
 and so
 $$
 [\mathcal{D}D_\psi^{\frac{1}{p}}]_{p}=
 [\mathcal{D}D^{\frac{1}{p}}]_{p}.
 $$
 It is clear that
 $$\A D_\psi^{\frac{1}{p}}\subset\A[\mathcal{D}D_\psi^{\frac{1}{p}}]_{p}=\A
 [\mathcal{D}D^{\frac{1}{p}}]_{p}\subset[\A D^{\frac{1}{p}}]_{p}
 $$
 and
 $$
 \A D^{\frac{1}{p}}\subset\A[\mathcal{D}D^{\frac{1}{p}}]_{p}=\A
 [\mathcal{D}D_\psi^{\frac{1}{p}}]_{p}\subset[\A D_\psi^{\frac{1}{p}}]_{p}.
 $$
 Therefore, we deduce that the first equality holds. We use similar method to prove  the second equality.
\end{proof}

Next,  we review the basic aspects of Hagrupp's reduction theorem, as described in \cite{H2,X}. Let $G=\cup_{n\ge1}2^{-n}\mathbb{Z}$  and let $\mathcal{R}$ denote
 $\mathcal{M}\rtimes_{\sigma^\varphi} G$  the discrete crossed product with respect to the modular automorphism group $\{\sigma_{t}^\varphi\}_{t\in \real}$.
 Consequently,  $\mathcal{R}$ emerges as a von Neumann algebra on the Hilbert space $\el_2(G,H)$, which is generated by the set of operators $\pi(x),\; x\in \M$ and $\lambda(t),\; s \in G$. These operators are  defined for any $\xi\in\el_2(G,H)$ and $t \in G$ by the following relations:
\be
(\pi(x)\xi)(t)=\sigma_{-t}(x)\xi(t),\quad (\lambda(s)\xi)(t)=\xi(t-s).
\ee
In this case $\pi$ gives a normal faithful representation of $\M$ on
$\el_2(G,H)$, so we identify $\pi(\M)$ with $\M$.  Furthermore, the operators $\pi(x)$ and $\lambda(t)$ fulfill the commutation relation delineated as
\beq\label{representation1}
\lambda(t)\pi(x)\lambda(t)^*=\pi(\sigma_{t}^\varphi(x)),\quad\forall t\in G,\;\forall x\in \M.
\eeq
 Let $\widehat{\varphi}$ be the dual weight of $\varphi$ on $\mathcal{R}$. Then $\widehat{\varphi}$ is again a faithful
normal state on  $\mathcal{R}$. It is restriction on $\mathcal{M}$ to coincide with  $\varphi$. There exists a unique normal faithful conditional expectation $\Phi$ mapping from  $\mathcal{R}$ onto  $\mathcal{M}$,  which  satisfies
$$
\widehat{\varphi}\circ\Phi = \widehat{\varphi},\quad\sigma_{t}^{\widehat{\varphi}}\circ\Phi=\Phi\circ\sigma_{t}^{\widehat{\varphi}}, \quad\forall t\in \real.
$$
According to Haagerup's reduction theorem (Theorem 2.1 and Lemma 2.7 in \cite{H2}), there is an
increasing sequence of von Neumann subalgebras  $\{\mathcal{R}_n\}_{n\ge1}$ within  $\mathcal{R}$, which satisfies  the following conditions:
\begin{enumerate}
\item For each  $n\in\mathbb{N}$, $\mathcal{R}_n$ is a finite von Neumann algebra;
\item  For every $n\in\mathbb{N}$, there  is a faithful normal conditional expectation from $\mathcal{R}$ onto $\mathcal{R}_n$ such that for any $t\in \real$,
$$
\hat{\varphi}\circ\Phi_n = \hat{\varphi},\quad  \sigma_{t}^{\widehat{\varphi}}\circ\Phi_n=\Phi_n\circ\sigma_{t}^{\widehat{\varphi}},\quad \Phi_n \circ\Phi_{n+1} = \Phi_n;
$$
\item For any $x\in\R$, $\Phi_n(x)$ converges to $x$  in the $\sigma$-strong  topology as $n\rightarrow\8$;
\item The $\bigcup_{n\ge1} \mathcal{R}_n$ is dense in $\mathcal{R}$ in the  $\sigma$-weak topology.
 \end{enumerate}

It is clear that $\mathcal{D}$ remains invariant under $\sigma^\varphi_t$. Hence,  the restriction $\sigma^\varphi_t|_{\mathcal{D}}$ of $\sigma^\varphi_t$ on $\mathcal{D}$  corresponds exactly to the modular automorphism group for the restricted state $\varphi|_{\mathcal{D}}$, there arises no necessity for differentiation between $\varphi$ and $\varphi|_{\mathcal{D}}$ or between  $\sigma^\varphi_t$and its restriction $\sigma^{\varphi}_t|_{\mathcal{D}}$, respectively.  We thus define
$$
\widehat{\mathcal{D}} = \mathcal{D} \rtimes_{\sigma^\varphi} G,
$$
$\widehat{\mathcal{D}}$ where is recognized as a von Neumann subalgebra of  $\mathcal{R}$, generated by all operators $\pi(x),\; x \in \mathcal{D}$ and
$\lambda(t),\; t \in G$. The expectation $\E$  extends to a normal faithful conditional expectation $\widehat{\mathcal{E}}$ from $\mathcal{R}$ onto $\widehat{\mathcal{D}}$, uniquely determined by the relation for all $x\in\M$ and $t\in G$:
\beq\label{conditional expectation}
\widehat{\mathcal{E}}(\lambda(t)x)=\lambda(t)\mathcal{E}(x).
\eeq
This expectation fulfills the compatibility condition
$$
\widehat{\mathcal{E}}\circ\Phi_n =\Phi_{n}\circ\widehat{\mathcal{E}},\quad \forall n\in\mathbb{N}.
$$
For natural number  $n\in\mathbb{N}$, we define $\mathcal{D}_n$ as the intersection of  $\mathcal{R}_n$ and  $\widehat{\mathcal{D}}$. It is essential to recognize that the restriction of $\Phi_n$ to  $\widehat{\mathcal{D}}$, and the restriction of $\widehat{\mathcal{E}}$
 to $\mathcal{R}_n$, both serve as normal, faithful conditional expectations onto $\mathcal{D}_n$ from  $\widehat{\mathcal{D}}$
 and $\mathcal{R}_n$, respectively.

Given that the invariance of  $\A$ under  $\sigma_{t}^{\varphi}$. By \eqref{representation1},  the set of all finite linear combinations of
$\lambda(t) \pi(x)$, with  $t \in G,\; x \in \A$, constitutes a subalgebra within  $\R$. We denote by $ \widehat{\mathcal{A}}$ the $\sigma$-weak closure of this set  in
$\mathcal{R}$, i.e.,
\beq\label{defi:Ahat}
\widehat{\mathcal{A}}=\overline{span\{\lambda(t) \pi(x):\;t \in G,\; x \in \A\}}^{\sigma-weakly }.
\eeq
For each  $n\in\mathbb{N}$, let $\mathcal{A}_n =\widehat{\mathcal{A}}\cap \mathcal{R}_n$. Referencing  Lemma 3.1-3.3 in \cite{X}, we get that $ \widehat{\mathcal{A}}$ (resp.  $\mathcal{A}_n$) is a maximal
subdiagonal subalgebra of $\mathcal{R}$ (resp. $\mathcal{R}_n$) with respect to $ \widehat{\mathcal{E}}$ (resp. $ \widehat{\mathcal{E}}|_{\mathcal{R}_n}$).

For Haagerup noncommutative
$H^p$-space, there is an analogue of \cite[ Theorem 3.1]{H2} (see \cite[Theorem 2.6]{BR1} ). We will use this result in next section.

 \begin{theorem}\label{thm:denseHp} Let $\M,\; \A,\;\E,\;\R,\;\widehat{\A},\;\widehat{\E},\;\Phi,\;\A_n,\;\Phi_n\; (\forall n\in \mathbb{N})$ be fixed as in the above and let $0 < p < \8$.
\begin{enumerate}
\item $\{H^p(\mathcal{A}_n)\}_{n\ge1}$ is a  increasing sequence of subspaces of $H^p(\widehat{\A}\:)$;
\item $\bigcup_{n\ge1}H^p(\mathcal{A}_n)$  is dense in  $H^p(\widehat{\A}\:)$;
\item For each $n$, $H^p(\mathcal{A}_n)$ is isometric to the usual
noncommutative $H^p$-space associated with a finite
subdiagonal algebra.
\item $H^p(\mathcal{A})$  and all   $H^p(\mathcal{A}_n)$
are 1-complemented in $H^p(\widehat{\A}\:)$ for $1\le p <\8$.
\end{enumerate}
 \end{theorem}

\section{ Complex interpolation theorem of the Haagerup
 $H^{p}$-spaces}

Let us  briefly revisit  the complex interpolation method that will be used herein. Consider the pair  $(X_0 ,X_1)$, which constitutes a couple for interpolation within the framework of quasi Banach spaces  be an interpolation couple of quasi Banach spaces. That is,  $X_j,\; j = 0,\;1$ are continuously embedded into a larger topological vector space $Y$, and
$X_0 \cap X_1$ is dense in $X_j,\; j = 0,\;1$. We reference $S$ (respectively, $\overline{S}$ ) to denote the open strip in the complex plane  $\mathbb{C}$, defined by $\{z:\;0<\mathrm{Re}z<1\}$ (respectively, the closed strip $\{z:\;0\le \mathrm{Re}z\le1\}$ )
We denote by $S$ (respectively, $\overline{S}$ ). Let
\be
 \F(X_0, X_1)=\left\{ f: \begin{array}{l}
 f : S \rightarrow X_0 + X_1\; \mbox{is a bounded, analytic }\\
 \mbox{function, which extend continuously to}\; \overline{S}  \\
 \mbox{ such that the traces}\; t\rightarrow f(j+it) \;  \mbox{ } \\
 \mbox{are bounded continuous functions}\\
 \mbox{ into}\; X_j,\; j = 0,\;1
\end{array}
\right\}.
 \ee
 The norm on $\F(X_0, X_1)$ is defined by
 \be
 \big\|f\big\|_{\F(X_0, X_1)}=\max\left\{\begin{array}{c}
 \sup_{t\in\real}\big\|f(it)\big\|_{X_0}\,,\;
 \sup_{t\in\real}\big\|f(1+it)\big\|_{X_1},\\
 \sup_{z\in S}\big\|f(z)\big\|_{X_0+X_1}
\end{array} \right\},
\ee
thereby endowing $\F(X_0, X_1)$ with a quasi Banach structure.   The complex interpolation space is defined as
\be
(X_0,X_1)_\theta =\left\{x:\; \exists\: f\in\F(X_0, X_1)\;\mbox{such that }\;f(\theta)=x\right\}\quad(0<\theta<1)
\ee
and
 \be
 \|x\|_{(X_0,\; X_1)_\theta}=\inf\big\{\big\|f\big\|_{\F(X_0, X_1)}\;:\; f(\theta)=x,\;
 f\in\F(X_0, X_1)\big\}.
 \ee
We recall that $ (X_0,X_1)_\theta$ is a quasi Banach space for $0<\theta<1$ (see \cite{KM}).

Furthermore, it is to be noted that the use of the third term is redundant when the maximum modulus principle is applicable.  However, this principle is not universal for all quasi Banach spaces, it is valid for Banach space. Hence, if $(X_0, X_1)$ is a compatible couple of complex Banach spaces, then the definition of the interpolation spaces here coincides with that of \cite{Ca}. It is well-known that the maximum  principle  fails for some quasi
Banach spaces (see \cite{Pe1}). There is
an important subclass of quasi Banach spaces called analytically convex,  as described in  \cite{K}, in which
the maximum modulus principle does hold. This convexity is characterized  via the condition:
\beq\label{eq:kalton}
\sup_{z\in S}\big\|f(z)\big\|_{X}\le C \max\{
 \sup_{t\in\real}\big\|f(it)\big\|_{X}\,,\;
 \sup_{t\in\real}\big\|f(1+it)\big\|_{X}\}
\eeq
for any analytic function $f: S\rightarrow X$ which is continuous on the closed strip
$\overline{S}$ (see \cite{KMM}).

The space $A(S)$ comprises complex valued functions that are analytic within $S$ and continuous and bounded
in  $\overline{S}$. The subset $A\F(X_0, X_1)$  includes functions that can be represented as finite sums of the form   $f(z)=\sum_{k=1}^n f_k(z)x_k$ with $f_k$ in $A(S)$ and $x_k$ in $X_0\cap X_1$.
Endowed  with the  (quasi)  norm
 \be
 \big\|f\big\|_{A\F(X_0, X_1)}=\max\left\{\begin{array}{c}
 \sup_{t\in\real}\big\|f(it)\big\|_{X_0}\,,\;
 \sup_{t\in\real}\big\|f(1+it)\big\|_{X_1},\\
 \sup_{z\in S}\big\|f(z)\big\|_{X_0+X_1}
\end{array} \right\},
\ee
$A\F(X_0, X_1)$ becomes a (quasi) normed  space. We define the
complex interpolation (quasi) norm $\|\cdot\|_\theta\;(0<\theta<1)$ on $X_0 \cap X_1$ as follows
 \be
 \|x\|_\theta=\inf\big\{\big\|f\big\|_{A\F(X_0, X_1)}\;:\; f(\theta)=x,\;
 f\in A\F(X_0, X_1)\big\}.
 \ee
We denote by $[X_0,X_1]_\theta$ the completion
of $(X_0 \cap X_1, \|\cdot\|_\theta)$.

Let $0<p_0<p_1<\8$. Then   $L^{p_0,\8}_0(0,\8)$ and  $L^{p_1,\8}_0(0,\8)$ are  separable symmetric quasi Banach function spaces on $(0,\8)$ and
$s_j$-convex for  $\frac{1}{p_j} < s_j < \8$, $j=0,\; 1$.  It is clear that
$$
L^{p_0,\8}_0(\N)+L^{p_1,\8}_0(\N)=(L^{p_0,\8}_0(0,\8)+L^{p_1,\8}_0(0,\8))(\N)
$$
 in the sense of equivalence of quasi norms (see \cite[Theorem 5]{BO1}).  Since  $L^{p_0,\8}_0(0,\8)+L^{p_1,\8}_0(0,\8)$ is  separable and  $s$-convex ($s=\min\{s_1,s_2\}$), we use the method in the proof of \cite[Lemma 1]{BO1} to obtain that there is an equivalent quasi norm $\|\cdot \|$ on $L^{p_0,\8}_0(\N)+L^{p_1,\8}_0(\N)$ which is plurisubharmonic. According to \cite[Lemma 4.5]{X1}, the  separability and  $s$-convexity of  $L^{p_0,\8}_0(0,\8) + L^{p_1,\8}_0(0,\8)$ imply that  $S(\N)$ is dense in $L^{p_0,\8}_0(\N) + L^{p_1,\8}_0(\N)$. Using the same method as  in the proof of \cite[Theorem 4.4]{X1},   we  obtain that
\beq\label{eq:densety-analytic function space}
 A \F(L^{p_0,\8}_0(\N),L^{p_1,\8}_0(\N))\quad \mbox{dense in}\quad \F(L^{p_0,\8}_0(\N),L^{p_1,\8}_0(\N)).
\eeq
 Hence, we have that
\beq\label{eq:equivalent-interpolation}
[L^{p_0,\8}_0(\N),L^{p_1,\8}_0(\N)]_\theta=(L^{p_0,\8}_0(\N),L^{p_1,\8}_0(\N))_\theta
\eeq
with equivalent norms.

We revisit the concept of embedding as it pertains to the spaces $L^{p_1}(\M)$ and $L^{p_0}(\M)$ , as described in \cite{Kos}.
Given  $0<p_0<p_1\le\8$,
we consider the embedding for any $0\leq\eta\leq1$as follows
$$
i_{p_1}^{\eta}:
x\mapsto D^{\frac{\eta}{r}}xD^{\frac{1-\eta}{r}},
$$
subject to the relation $\frac{1}{p_1}+\frac{1}{r}=\frac{1}{p_0}$. Here, we define  $L^{p_1}(\M)^\eta$ as
$$
L^{p_1}(\M)^\eta=i_{p_1}^{\eta}(L^{p_1}(\M))=D^{\frac{\eta}{r}}L^{p_1}(\M)D^{\frac{1-\eta}{r}}.
$$
For $ x\in L^{p_1}(\M)$, the norms are given by
\be
 \|i_{p_1}^{\eta}(x)\|_{p_1}^{\eta}=
     \|D^{\frac{\eta}{r}}xD^{\frac{1-\eta}{r}}\|_{p_1}^{\eta}=\|x\|_{p_1},\quad\|x\|_{p_0}^\eta=\|D^{\frac{\eta}{r}}xD^{\frac{1-\eta}{r}}\|_{p_0}.
\ee
In the following, we do not consider the powers of the norms $\|\cdot\|_{p_1}$ or $\|\cdot\|_{p_0}$. Therefore, the notations used above will not lead to any confusion.
According to
$\|\cdot\|_{p_0}^{\eta}\leq\|\cdot\|_{p_1}^{\eta}$ within  the subspace $L^{p_1}(\M)^\eta$
of $L^{p_0}(\M)$,  we get the compatibility of the pair $(L^{p_1}(\M)^\eta, L^{p_0}(\M))$.

For $p_0<p\le p_1$ and $0\leq\eta\leq1$,  we define the imbedding
$$
i_{p}^{\eta}: L^{p}(\M)\hookrightarrow L^{p_0}(\M)
$$
via the mapping
$x\mapsto D^{\frac{\eta}{q}}xD^{\frac{1-\eta}{q}}$,
where $\frac{1}{p}+\frac{1}{q}=\frac{1}{p_0}$. Consequently, we assert that
$$
i_{p}^{\eta}(
L^{p}(\M))=D^{\frac{\eta}{q}}L^{p}(\M)D^{\frac{1-\eta}{q}}\subseteq L^{p_0}(\M).
$$
We express the norm as
$$
 \|i_{p}^{\eta}(x)\|_{p}^{\eta}=
 \|D^{\frac{\eta}{q}}xD^{\frac{1-\eta}{q}}\|_{p}^{\eta}=\|x\|_{p},\quad\forall
 x\in L^{p}(\mathcal{M}).
 $$
  It is clear that $L^{p_1}(\M)^\eta+ L^{p_0}(\M)=L^{p_0}(\M)$ with equivalent norm.  Using 
Theorem \ref{thm:Lp-isometric weak Lp}, we get that $L^{p_1}(\M)^\eta+ L^{p_0}(\M)$ is isomorphic to closed subspace of $L^{p_0,\8}_0(\N)$. Hence, by \cite[Proposition 7.5]{KMM}, $L^{p_1}(\M)^\eta+ L^{p_0}(\M)$ is analytically convex. Similar to \eqref{eq:equivalent-interpolation}, it is deduced that
\beq\label{eq:equivalent-interpolation-Lp}
(L^{p_1}(\M)^\eta,L^{p_0}(\M))_\theta=[L^{p_1}(\M)^\eta,L^{p_0}(\M)]_\theta
\eeq
with equivalent norms.

In accordance with \cite[theorem 4.1]{GYZ}, for  $0<p_0<p_1\le\8$, it is established that
\beq\label{eq:interpolation-Lp}
[L^{p_1}(\M)^\eta, L^{p_0}(\M)]_\theta=i_{p_\theta}^{\eta}(L^{p_\theta}(\M)),
\eeq
 where
$\frac{1}{p_\theta}=\frac{1-\theta}{p_0}+\frac{\theta}{p_1}$ (refer to \cite[Theorem 9.1]{Kos} for the case where $1\le p_0<p_1\le\8$).

Let $H^{p_1}(\A)^\eta= i_{p_1}^{\eta}(H^{p_1}(\A))=D^{\frac{\eta}{q}}H^{p_1}(\A)D^{\frac{1-\eta}{q}}$ where $\frac{1}{p_1}+\frac{1}{q}=\frac{1}{p_0}$. Then $H^{p_1}(\A)^\eta\subset H^{p_0}(\A)$  and
$(H^{p_1}(\A)^\eta, H^{p_0}(\mathcal{A}))$ is compatible. Since $H^{p_1}(\A)^\eta+ H^{p_0}(\mathcal{A})=H^{p_0}(\mathcal{A})$ with equivalent norm and $H^{p_0}(\mathcal{A}))$ is closed subspace of $L^{p_0}(\M)$,   using  \cite[Proposition 7.5]{KMM}, similar to \eqref{eq:densety-analytic function space}, we obtain  that 
\beq\label{eq:densety-analytic function space-Hp}
 A \F(H^{p_1}(\A)^\eta, H^{p_0}(\mathcal{A}))\quad \mbox{dense in}\quad \F(H^{p_1}(\A)^\eta, H^{p_0}(\mathcal{A})).
\eeq

The next result is known. It follows from Pisier/Xu's theorem (see \cite{PX} or \cite[Theorem 1.1]{BO})
\begin{proposition} \label{lem:pisier-xu} Let $\M$ be a  finite von Neumann algebra,  $\A$ be a  subdiagonal algebra of $\M$, $0<p_0,\ p_1\le\8$, $0<\theta<1$ and $\frac{1}{p_\theta}=\frac{1-\theta}{p_0}+\frac{\theta}{p_1}$. If
$f\in A\F(H^{p_1}(\A), H^{p_0}(\A))$ and
\be
M_k=\sup_{t\in\mathbb{R}}\|f(k+it)\|_{p_k}, \qquad k=0,1,
\ee
then $\|f(\theta)\|_{p_\theta}\le M_0^{1-\theta}M_1^{\theta}$.
\end{proposition}

The following result is an $H^p$-space analogue of \cite[Proposition 3.6]{GYZ}.  The proof differs slightly from that of
\cite[Proposition 3.6]{GYZ}.
\begin{proposition} \label{pro:trueline} Let $0\leq\eta\leq1$, $0<p_0,\ p_1<\8$ and $0<\theta<1$. If
$f\in A\F(H^{p_1}(\A)^\eta, H^{p_0}(\A))$ and
\be
M_k=\sup_{t\in\mathbb{R}}\|f(k+it)\|_{p_k}, \qquad k=0,1,
\ee
then $\|f(\theta)\|_{p_\theta}\le M_0^{1-\theta}M_1^{\theta}$.
\end{proposition}
\begin{proof} Let $\widehat{\A}$ be as in Theorem \ref{thm:denseHp}. By \cite[Proposition 2.3]{BR1}, $H^{p}(\A)$ coincides isometrically with a subspace of $H^{p}(\widehat{\A})$  for any $0<p<\8$. Hence, it is sufficient to prove that if
$$f\in A\F( H^{p_1}(\widehat{\A})^\eta,H^{p_0}(\widehat{\A}))$$
with $\sup_{t\in\mathbb{R}}\|f(k+it)\|_{p_k}\le 1 \;(k=0,1)$,
 then $\|f(\theta)\|_{p_\theta}\le 1$.

 Let
$$
f(z)=\sum_{k=1}^n f_k(z)x_k,\qquad f_k\in A(S),\qquad x_k\in H^{p_1}(\widehat{\A})^\eta.
$$
By Theorem \ref{thm:denseHp}, $\{H^{p_1}(\A_n)\}_{n\ge1}$ is a  increasing sequence of subspaces of $H^{p_1}(\widehat{\A})$ and $\bigcup_{n\ge1}H^{p_1}(\A_n)$  is dense in  $H^{p_1}(\widehat{\A})$. Hence, for $\varepsilon > 0$, there exist an natural number  $N$ and  an operator $x\in H^{p_1}(\A_N)$ such that
$\|f(\theta)-x\|_{p_1}<\frac{\varepsilon}{2}$.  We consider
 \be
g(z)=f(z)-(f(\theta)-i_{p_1}^{\eta}(x))=\sum_{k=1}^n g_k(z)i_{p_1}^{\eta}(y_k),
\ee where $g_k\in A(S)$ and $y_k\in H^{p_1}(\widehat{\A})$.
Therefore, there exist an natural number  $N'$ and  a sequence
$\{a_k\}\subset H^{p_1}(\A_{N'})$
such that for each $k$
\be
\|y_k-a_k\|_{p_1}<\frac{\varepsilon}{2^{k+1}C_k},
\ee
where $C_k=\sup_{z\in\overline{S}}|g_k(z)|<\8$. Set
\be
h(z)=\sum_{k=1}^n g_k(z)a_k,
\qquad
l(z)=\sum_{k=1}^n g_k(z)(y_k-a_k).
\ee
If $p_0\ge1$ (respectively, $0<p_0<1$), then for any $z\in\overline{S}$,
\be\begin{array}{rl}
\|l(z)\|_{p_0}&\le \sum_{k=1}^n|g_k(z)|\|i_{p_1}^{\eta}(y_k-a_k)\|_{p_0}\\
&\le \sum_{k=1}^nC_k\|y_k-a_k\|_{p_1}<\sum_{k=1}^n\frac{\varepsilon}{2^{k+1}}
\end{array}
\ee
\be
(\mbox{respectively,}\;\|l(z)\|_{p_0}^{p_0}\le \sum_{k=1}^n\frac{\varepsilon^{p_0}}{2^{p_0(k+1)}})
\ee
and
\be\begin{array}{rl}
\|h(it)\|_{p_0}&\le \|g(it)-l(it)\|_{p_0}\le \|g(it)\|_{p_0}+\|l(it)\|_{p_0}\\
&=\|f(it)-(f(\theta)-i_{p_1}^{\eta}(x))\|_{p_0}+\|l(it)\|_{p_0}\\ &\le\|f(it)\|_{p_0}+\frac{\varepsilon}{2}+\|l(it)\|_{p_0} \\
&\le1+\frac{\varepsilon}{2}+\sum_{k=1}^n\frac{\varepsilon}{2^{k+1}}=1+O(\varepsilon)
\end{array}
\ee
\be
(\mbox{respectively,}\;\|h(it)\|_{p_0}^{p_0}\le1+\frac{\varepsilon^{p_0}}{2^{p_0}}+ \sum_{k=1}^n\frac{\varepsilon^{p_0}}{2^{p_0(k+1)}})\le 1+O(\varepsilon^{p_0})).
\ee
whence $\|h(it)\|_{p_0}\le1+O(\varepsilon)$. Similarly, we get that
 $$
 \|h(1+it)\|_{p_1}\le1+ O(\varepsilon).
 $$
 Using Proposition \ref{lem:pisier-xu}, we deduce that
 $\|h(\theta)\|_{p_{\theta}}\le1+O(\varepsilon)$. Hence, if $p_0\ge1$ (respectively, $0<p_0<1$), then
 \be\begin{array}{rl}
\|f(\theta)\|_{p_\theta}&=\|g(\theta)+(f(\theta)-i_{p_1}^{\eta}(x))\|_{p_\theta}\le\|g(\theta)\|_{p_\theta}+\|f(\theta)-x\|_{p_1}\\
&\le \|h(\theta)+l(\theta)\|_{p_\theta}+\frac{\varepsilon}{2}\le \|h(\theta)\|_{p_\theta}+\|l(\theta)\|_{p_\theta}+\frac{\varepsilon}{2}\\
 &\le\|h(\theta)\|_{p_\theta}+\sum_{k=1}^n\frac{\varepsilon}{2^{k+1}}+\frac{\varepsilon}{2}=1+O(\varepsilon)
\end{array}
\ee
\be
(\mbox{respectively,}\;\|f(\theta)\|_{p_\theta}^{p_\theta}\le 1+O(\varepsilon^{p_\theta})).
\ee
Letting $\varepsilon\rightarrow0$, we obtain the desired result.

\end{proof}




\begin{theorem} \label{thm:interpolationhp-hat} Let $0<p_0<p_1\le\8$ and $0<\theta<1$. Then
\be
(H^{p_1}(\widehat{\A}\:)^\eta,H^{p_0}(\widehat{\A}\:))_\theta=i_{p_\theta}^{\eta}( H^{p_\theta}(\widehat{\A}\:)),
\ee
 where $\frac{1}{p_\theta}=\frac{1-\theta}{p_0}+\frac{\theta}{p_1}$.
\end{theorem}
\begin{proof} The result for the case $ p_0\ge1$ has been previously established in  \cite[Lemma 3.2]{BR1}. We now turn our attention to the case when  $0<p_0 < 1$. First, we assume that $ p_1<\8$. Our objective is to demonstrate that
\be
[H^{p_1}(\widehat{\A}\:)^\eta,H^{p_0}(\widehat{\A}\:)]_\theta\subset i_{p_\theta}^{\eta}( H^{p_\theta}(\widehat{\A}\:)).
\ee
To substantiate this inclusion, it suffices to verify that
\beq\label{eq:interpolation1}
\|i_{p_\theta}^{\eta}(x)\|_{p_\theta}\le \|i_{p_\theta}^{\eta}(x)\|_\theta\qquad \forall x\in H^{p_1}(\widehat{\A}\:)^\eta.
\eeq
Let $f\in A\F(H^{p_1}(\widehat{\A}\:)^\eta,H^{p_0}(\widehat{\A}\:))$ such that
$f(\theta) = x$ and
\be
\|f\|_{A\F(H^{p_1}(\widehat{\A}\:)^\eta,H^{p_0}(\widehat{\A}\:))}\le1.
\ee
From the proof of Proposition \ref{pro:trueline}, we have that
$$
\|i_{p_\theta}^{\eta}(x)\|_{p_\theta}=\|f(\theta)\|_{p_\theta}\le1,
$$
confirming the validity of \eqref{eq:interpolation1}.

Next, we aim to establish that
\be
 i_{p_\theta}^{\eta}( H^{p_\theta}(\widehat{\A}\:))\subset[H^{p_1}(\widehat{\A}\:)^\eta,H^{p_0}(\widehat{\A}\:)]_\theta.
\ee
For this purpose, it is necessary to find a constant $C>0$ for which every  $x\in H^{p_\theta}(\widehat{\A}\:)$ satisfies
\be
\|i_{p_\theta}^{\eta}(x)\|_\theta\le C\|i_{p_\theta}^{\eta}(x)\|_{p_\theta}.
\ee
Consider $x\in H^{p_\theta}(\widehat{\A}\:)$ and a sequence $\{\mathcal{A}_n\}_{n\ge1}$ as described in \ref{thm:denseHp}. Given that  $\A_n$ is a subdiagonal subalgebra of the finite von Neumann algebra $\R_n$, it follows that
$$
 [H^{p_1}(\A_n),
H^{p_0}(\A_n)]_{\theta}=H^{p_\theta}(\A_n)
 $$
with equivalent norms for all  $ n\in \mathbb{N}$ (refer to  the remark following Lemma 8.5 in \cite{PX} or Theorem 1.1 in \cite{BO}). By virtue of (1) of Theorem \ref{thm:denseHp}, we assert that
 $$
 H^{p_\theta}(\A_n)=(H^{p_1}(\A_n),
H^{p_0}(\A_n))_{\theta}\subset(H^{p_1}(\widehat{\A}\:)^\eta, H^{p_0}(\widehat{\A}\:))_\theta
 $$
 and
 \beq\label{eq:inequality2}
  \|a\|_{\theta}\le C\|a\|_{p_\theta},\qquad \forall a\in H^{p_\theta}(\A_n)
 \eeq
with $C$  depends only on  $\theta$.

Applying (2) of Theorem \ref{thm:denseHp}, we acquire a convergent sequence $\{a_n\}_{n\ge1}$ within $\bigcup_{n\ge1}H^{p_\theta}(\A_n)$ such that
 $$
\lim_{n\rightarrow\8}\| a_n -x\|_{p_\theta}=0.
 $$
The monotonicity of $\{H^{p_\theta}(\mathcal{A}_n)\}_{n\ge1}$ coupled with \eqref{eq:inequality2} allows us to deduce that $\{a_n\}_{n\ge1}$ forms a Cauchy sequence in
$(H^{p_1}(\widehat{\A}\:)^{\eta}, H^{p_0}(\widehat{\A}\:))_{\theta}$.
Therefore, there exists $y$ within $(H^{p_1}(\widehat{\A}\:)^{\eta}, H^{p_0}(\widehat{\A}\:))_{\theta}$ satisfying
$$
\lim_{n\rightarrow\8}\|a_n-y\|_{ \theta}=0.
$$
Invoking \eqref{eq:interpolation1}, we conclude that $\lim_{n\rightarrow\8}\|a_n-y\|_{p_\theta}=0$, and hence $y=x$. This establishes that
$$
  \|x\|_{\theta}\le C\|x\|_{p_\theta},\qquad \forall x\in H^{p_\theta}(\widehat{\A}\:).
$$
As a result, we affirm that $ (H^{p_1}(\widehat{\A}\:)^{\eta}, H^{p_0}(\widehat{\A}\:))_{\theta}= i_{p_\theta}^{\eta}(H^{p_\theta}(\widehat{\A}\:))$.

In the case $p_1 = \8$, we use \cite[Lemma 1]{W} (which is also true for quasi Banach spaces) in combination with the result from the first case and \cite[Lemma 3.2]{BR1} in a similar way to obtain the desired result.

\end{proof}

Let $\mathcal{A}$ be a maximal subdiagonal subalgebra in $\mathcal{M}$ with
respect to $\mathcal{E}$. Let $x\in\mathcal{M}$. If the function $t\mapsto
\sigma_{t}(x)$ extends to an analytic function from $\mathbb{C}$ to
$\mathcal{M}$, then we call $x$ an analytic vector in $\mathcal{M}$.
We denote by  $\mathcal{M}_{a}$  the family of all analytic vectors in $\mathcal{M}$. Then $\mathcal{M}_{a}$ is a w*-dense $\ast$-subalgebra of
$\mathcal{M}$ (cf. \cite{PT}). The next result is proved in \cite[Lemma 2.3]{BR} for the case $1\le p<\8$ (see the proof of Theorem 2.5 in \cite{J1}), we extend it to the case $0<p<1$.

\begin{lemma}\label{lem:analytic}
Let $\mathcal{A}_{a}$ be respectively the families of all analytic
vectors in $\mathcal{A}$.  If $0< p<1$ and $0\le\eta\le1$, then
$D^{\frac{\eta}{p}}\mathcal{A}_{a}D^{\frac{1-\eta}{p}}$ is dense in $H^{p}(\mathcal{A})$.
\end{lemma}
\begin{proof}
 Choose $r>0$  such that $\frac{1}{r}+1=\frac{1}{p}$. By \cite[Lemma 2.3]{BR}, we get
 $$
 \begin{array}{rl}
 [D^{\frac{\eta}{p}}\mathcal{A}_{a}D^{\frac{1-\eta}{p}}]_p&=[D^{\frac{\eta}{r}}(D^{\eta}\mathcal{A}_{a}D^{1-\eta})D^{\frac{1-\eta}{r}}]_p \\
 &=[D^{\frac{\eta}{r}}[D^{\eta}\mathcal{A}_{a}D^{1-\eta}]_1D^{\frac{1-\eta}{r}}]_p\\
      & =[D^{\frac{\eta}{r}}H^{1}(\mathcal{A})D^{\frac{1-\eta}{r}}]_p=H^{p}(\mathcal{A}).
 \end{array}
 $$
\end{proof}

We define the complex interpolation spaces $H^{p_\theta}(\widehat{\A}\:)_{L}$ and $H^{p_\theta}(\widehat{\A}\:)_{R}$ as follows:
$$
H^{p_\theta}(\widehat{\A}\:)_{L}=(H^{p_1}(\widehat{\A}\:)^{0}, H^{p_0}(\widehat{\A}\:))_{\theta}$$
and
$$
H^{p_\theta}(\widehat{\A}\:)_{R}=(H^{p_1}(\widehat{\A}\:)^1,H^{p_0}(\widehat{\A}\:))_{\theta}.
$$
These are referred to as the left and the right  $H^{p_\theta}$-spaces with respect to $D$, respectively.  According to Theorem \ref{thm:interpolationhp-hat}, it is established that
$$
H^{p_\theta}(\widehat{\A}\:)_{L}=H^{p_\theta}(\widehat{\A}\:)D^{\frac{1}{r}},\qquad H^{p_\theta}(\widehat{\A}\:)_{R}=D^{\frac{1}{r}}H^{p_\theta}(\widehat{\A}\:),
$$
where $\frac{1}{r}+\frac{1}{p_\theta}=\frac{1}{p_0}$.
We present a Stein-Weiss type interpolation theorem as follows

\begin{theorem}  Let $0<p_0<p<\8$, $\frac{1}{p}+\frac{1}{q}=\frac{1}{p_0},\;0\le\eta\le1$. Then
$$
(H^{p}(\widehat{\A}\:)D^{\frac{1}{q}},D^{\frac{1}{q}}H^{p}(\widehat{\A}\:))_{\eta}=
D^{\frac{\eta}{q}}H^{p}(\widehat{\A}\:)D^{\frac{1-\eta}{q}}.
$$
This implies that
$$
(H^{p}(\widehat{\A}\:)_{L},H^{p}(\widehat{\A}\:)_{R})_{\eta}=
(H^{p_1}(\widehat{\A}\:)^\eta,H^{p_0}(\widehat{\A}\:))_{\eta}.
$$
\end{theorem}

\begin{proof}
Assume that $f\in A\F(H^{p}(\widehat{\A}\:)D^{\frac{1}{q}},D^{\frac{1}{q}}H^{p}(\widehat{\A}\:))$ with
$f(\eta) = x$ and
\be
\|f\|_{A\F(H^{p_1}(\widehat{\A}\:)^\eta,H^{p_0}(\widehat{\A}\:))}\le1.
\ee
Using the method in the proof of Proposition \ref{pro:trueline}, we get that
$$
\|i_{p}^{\eta}(x)\|_{p}=\|f(\eta)\|_{p}\le1,
$$
 So, we deduce that
\beq\label{inclusion-1}
(H^{p}(\widehat{\A}\:)D^{\frac{1}{q}},D^{\frac{1}{q}}H^{p}(\widehat{\A}\:))_{\eta}\subset
D^{\frac{\eta}{q}}H^{p}(\widehat{\A}\:)D^{\frac{1-\eta}{q}}.
\eeq

Conversely, first assume that  $x\in\widehat{\A}_aD^{\frac{1}{p}}$.  For $\delta>0$,  we define that
$$
f(z)=e^{\delta(z^2-\eta^2)}D^{\frac{z}{q}}xD^{\frac{1-z}{q}}, \qquad 0\leq\mbox{Re}z\leq1.
$$
Clearly,
$$
 f(it)\in H^{p}(\widehat{\A}\:)D^{\frac{1}{q}},\quad\| f(it)\|_{p}\le\|x\|_{p},
 $$
 $$
 f(1+it)\in D^{\frac{1}{q}}H^{p}(\widehat{\A}\:),\quad \| f(1+it)\|_{p}\leq e^{\delta(1-\eta^2)}\|x\|_{p}.
 $$
 Hence,
  $$
 f(z)\in
 F(H^{p}(\widehat{\A}\:)D^{\frac{1}{q}},D^{\frac{1}{q}}H^{p}(\widehat{\A}\:))
 $$
 and
 $$
 \||f|\|\le e^{\delta(1-\eta^2)}\|x\|_{p},\quad f(\eta)=D^{\frac{\eta}{q}}xD^{\frac{1-\eta}{q}}.
$$
It follows that
\be
\|i_{p_\eta}^{\eta}(x)\|_\theta\le e^{\delta(1-\eta^2)}\|i_{p}^{\eta}(x)\|_{p}.
\ee
Letting $\delta\rightarrow0$, we obtain that
\be
\|i_{p}^{\eta}(x)\|_\theta\le\|i_{p_\theta}^{\eta}(x)\|_{p}.
\ee
Therefore,  $i_{p}^{\eta}(x)\in  (H^{p}(\widehat{\A}\:)D^{\frac{1}{q}},D^{\frac{1}{q}}H^{p}(\widehat{\A}\:))_{\eta}$. Using  Lemma \ref{lem:analytic}, we get that
$$
i_{p}^{\eta}(x)\in  (H^{p}(\widehat{\A}\:)D^{\frac{1}{q}},D^{\frac{1}{q}}H^{p}(\widehat{\A}\:))_{\eta}, \qquad \forall x\in H^{p}(\widehat{\A}\:).
$$
So, we deduce that
\beq\label{inclusion-2}
D^{\frac{\eta}{q}}H^{p}(\widehat{\A}\:)D^{\frac{1-\eta}{q}}\subset (H^{p}(\widehat{\A}\:)D^{\frac{1}{q}},D^{\frac{1}{q}}H^{p}(\widehat{\A}\:))_{\eta}.
\eeq
By \eqref{inclusion-1} and \eqref{inclusion-2}, we obtain the desired result. The second result follows immediately by virtue of Theorem \ref{thm:interpolationhp-hat}.
\end{proof}

\section{ Complex interpolation theorem  for the type 1 subdiagonal algebras case}

We embedding $\mathcal{A}_a$ into $H^{p}(\mathcal{A})$ via the map $x\mapsto xD^\frac{1}{p}$ to reformulate the complex interpolation $(H^{p_0}(\mathcal{A}), H^{p_1}(\mathcal{A}))_\theta$ as the following way. Let $0<p_0<p_1\le\8$, $0\le\eta\le1$ and $0<\theta<1$.  We denote by
$A\F_a$
the family of functions of the form $f(z)=\sum_{k=1}^n f_k(z)x_k$ with $f_k$ in $A(S)$ and $x_k$ in $\mathcal{A}_a$.
Equipped with the (quasi) norm
 \be
 \big\|f\big\|_{A\F_a(p_1,p_0)}=\max\left\{
 \sup_{t\in\real}\big\|f(it)D^{\frac{1}{p_0}}\big\|_{p_0}\,,\;
 \sup_{t\in\real}\big\|i_{p_1}^{\eta}(f(1+it)D^{\frac{1}{p_1}})\big\|_{p_1}\right\},
\ee
$A\F_a$ becomes a (quasi) Banach space. We define the
 norm \\ $\|\cdot\|_{\theta,a}\;(0<\theta<1)$ on $\mathcal{A}_{a}D^{\frac{1}{p_\theta}}$ ($\frac{1}{p_\theta}=\frac{1-\theta}{p_0}+\frac{\theta}{p_1}$) as follows: if $x\in\mathcal{A}_{a}$, then
 \be
 \|i_{p_\theta}^{\eta}(xD^{\frac{1}{p_\theta}})\|_{\theta,a}=\inf\big\{\big\|f\big\|_{A\F_a(p_1,p_0)}\;:\; f(\theta)=x,\;
 f\in A\F_a\big\}.
 \ee
By \eqref{eq:densety-analytic function space-Hp} and Lemma \ref{lem:analytic},
\be
\mbox{the completion
of}\; (\mathcal{A}_{a}D^{\frac{1}{p_\theta}}, \|\cdot\|_{\theta,a})=
(H^{p_1}(\A)^\eta,H^{p_0}(\A))_\theta.
\ee
Since $i_{p_1}^{\eta}(\mathcal{A}_{a}D^{\frac{1}{p_1}})=\mathcal{A}_{a}D^{\frac{1}{p_0}}$ and $i_{p_\theta}^{\eta}(\mathcal{A}_{a}D^{\frac{1}{p_\theta}})=\mathcal{A}_{a}D^{\frac{1}{p_0}}$, 
\beq\label{eq:analytic-inter}
(H^{p_1}(\A)^\eta,H^{p_0}(\A))_\theta=(H^{p_1}(\A)^0,H^{p_0}(\A))_\theta
\eeq 
with isometric norms.

\begin{lemma}\label{lem:convergence-sequence}
  Let $0 < p < \8$, $(x_n)_{n\ge1}\subset\M$ be a sequence  and $x\in\M$. If $(x_nD^{\frac{1}{p_0}})_{n\ge1}$ converges to $xD^{\frac{1}{p_0}}$  in norm for some $0<p_0<\8$,
then
this is true  for all $0 < p < \8$.
\end{lemma}
\begin{proof}
  By \eqref{eq:generalized singular-haagerup}, for any $t>0$, we have that
  $$
  \mu_t(x_nD^{\frac{1}{p_0}}-xD^{\frac{1}{p_0}})=t^{-\frac{1}{p_0}}\|x_nD^{\frac{1}{p_0}}-xD^{\frac{1}{p_0}}\|_{p_0}\rightarrow0\qquad\mbox{as}\quad n\rightarrow\8.  
  $$
Using \cite[Lemma 3.1]{FK}, we get $(x_nD^{\frac{1}{p_0}})_{n\ge1}$ converges to $xD^{\frac{1}{p_0}}$ in the measure topology. Set $e_n=e_{(\frac{1}{n},\8)}(D)$, for any $n\in\mathbb{N}$. Then $e_n\in\N$ and $(e_n)_{n\ge1}$ increases strongly to $1$, and so $\tau(e_n^\perp)\rightarrow0$ as $n\rightarrow\8$. Let $f_m(\lambda)=\lambda^{-\frac{1}{p_0}}\chi_{(\frac{1}{m},\8)}$, for any $m\in\mathbb{N}$. Then $f_m(D)\in\N$ and 
$$
x_ne_m=x_nD^{\frac{1}{p_0}}f_m(D)\rightarrow xD^{\frac{1}{p_0}}f_m(D)=xe_m
$$
in the measure topology. Let $\varepsilon> 0, \delta> 0$. Choose $m_0\in \mathbb{N}$ such that $\tau(e_{m_0}^\perp)<\frac{\delta}{2}$. Since $x_ne_{m_0}\rightarrow xe_{m_0}$
in the measure topology, there exists $n_0\in \mathbb{N}$ and a projection $f\in\N$ such that $\tau(f^\perp)<\frac{\delta}{2}$ and
$$
\|(xe_{m_0}-x_ne_{m_0})f\|<\varepsilon,\qquad\mbox{for any }\quad n\ge n_0.
$$
Let $e = e_{m_0}\wedge f$. Then $\tau(e^\perp)<\delta$ and for any $n\ge n_0$,
$$
\|(x-x_n)e\|=\|(xe_{m_0}-x_ne_{m_0})e\|=\|(xe_{m_0}-x_ne_{m_0})fe\|<\varepsilon.
$$
Therefore, $x_n\rightarrow x$ in the measure topology. Applying  \eqref{eq:generalized singular-haagerup} and  \cite[Lemma 2.5 (vii), Lemma 3.1]{FK}, we deduce that for any $0<p<\8$,
$$\begin{array}{rl}
    \|x_nD^{\frac{1}{p}}-xD^{\frac{1}{p}}\|_{p}&=\mu_1(x_nD^{\frac{1}{p}}-xD^{\frac{1}{p}}) \\
     & \le \mu_{\frac{1}{2}}(x_n-x)\mu_{\frac{1}{2}}(D^{\frac{1}{p}}) \\
     &  = \mu_{\frac{1}{2}}(x_n-x)2^{\frac{1}{p}}.
  \end{array}
$$
Thus  $(x_nD^{\frac{1}{p}})_{n\ge1}$ converges to $xD^{\frac{1}{p}}$  in norm.
\end{proof}

We recall that a closed subspace $K$
of $L^2(\mathcal{M})$ is called a right (resp. left) $\mathcal{A}$-invariant subspace of $L^2(\mathcal{M})$ if $K\mathcal{A}\subset K$ (resp. $\mathcal{A}K\subset K$).

Let $K$ be a right $\mathcal{A}$-invariant subspace of $L^{2}({\mathcal{M}})$. We call
$$
W=K\ominus[K\mathcal{A}_{0}]_{2}
$$
is  the right wandering subspace of $K$. We say that  $K$ is  of type 1 if $W$ generates $K$ as an
$\mathcal{A}$-module (that is $K=[WA]_{2})$ and   $K$ is  type 2 if $W=0$. See \cite{BL3,L2} for more details.

Following \cite{J2,ZJ}, we call the $\mathcal{A}$ is a type 1 subdiagonal subalgebra of $\mathcal{M}$, if every right $\mathcal{A}$-invariant subspace in $H^2(\mathcal{A})$ is of type 1.

 Let $0<p\le\8$. We call $h\in H^p(\A)$ is  right outer or
left outer  according to
$[h\mathcal{A}]_p=H^p(\mathcal{A})$ or $[\mathcal{A}h]_p=H^p(\mathcal{A})$.

\begin{theorem}\label{thm:reisz} Let $\mathcal{A}$ be a type 1 subdiagonal subalgebra.   If $n\in\mathbb{N}$ and $0<p<1$ such that $np\ge1$, then for any $h\in H^p(\mathcal{A})$ and $\varepsilon>0$, there exist $h_k\in H^{np}(\mathcal{A})$
$(k=1,2,\cdots, n)$ and a contraction $a\in\A$ such that
$$
h=ah_1h_2\cdots h_n,\quad\|h_k\|_{np}\le\|h\|_p^{\frac{1}{n}}+\varepsilon\quad(k=1,2,\cdots, n).
$$
\end{theorem}
\begin{proof} Let $h=u|h|$ be the
polar decomposition of $h$. Then
 $$
  h=v|h|^{1/n} |h|^{1/n}\cdots\, |h|^{1/n}.
 $$
 It is clear that
 $|h|^{1/n}\in L^{np}(\mathcal{M})$. By \cite[Theorem 2.3]{ZJ}, we have a
factorization
 $|h|^{1/n}=v_n h_n$
with $v_n\in\mathcal{M}$ a contraction, $h_n\in
H^{np}(\mathcal{A})$ a right outer such that
$$
\|h_n\|_{np}<\||h|^{1/n}\|_{np}+\varepsilon=\|h\|_p^{\frac{1}{n}}+\varepsilon.
$$
 Repeating this argument, we again get
a same factorization of $|h|^{1/n}v_n$:
 $$
 |h|^{1/n}v_n=v_{n-1} h_{n-1};
 $$
and then for $|h|^{1/n}v_{n-1}$, and so on. In this way, we obtain a
factorization:
 $$h=a h_1\cdots h_n,$$
where $a\in\mathcal{M}$ is a contraction, $h_k\in
H^{np}(\mathcal{A})\;(k=1,2,\cdots, n)$ is a right outer. Since $a h_1\cdots h_n\mathcal{A}\subset H^{p}(\mathcal{A})$, we
then see that
$$
aH^{p}(\mathcal{A})\subset H^{p}(\mathcal{A}).
$$
Hence, $aD^{\frac{1}{p}}\in H^{p}(\mathcal{A})$. By Proposition \ref{lem:hp}, there a sequence  $(a_n)_{n\ge1}\subset\A$ such that $(a_nD^{\frac{1}{p}})_{n\ge1}$ converges to $aD^{\frac{1}{p}}$  in norm in $L^p(\M)$. Using Lemma \ref{lem:convergence-sequence}, we get that  $(a_nD^{\frac{1}{2}})_{n\ge1}$ converges to $aD^{\frac{1}{2}}$  in norm in norm in $L^2(\M)$. This gives $aD^{\frac{1}{2}}\in H^{2}(\mathcal{A})$. Applying \cite[Proposition 2.8(1)]{BR1}, we obtain that $a\in\A$.
\end{proof}

Let $0<p,q\le\8$. We define 
$$
 H^p(\A)\odot H^q(\A)=\{x:\;x=yz,\;y\in H^p(\A),\;z\in H^q(\A)\}
 $$ 
 and 
 $$\|x\|_{H^p(\A)\odot H^q(\A)}=\inf\big\{\|y\|_p\,\|z\|_q\;:\; x=yz,\; y\in H^p(\A),\;
 z\in H^q(\A) \big\}.$$

\begin{theorem}
   Let $0< p, q, r\le\8$ such that $1/p=1/q+1/r$.
Then for $h\in H^p(\A)$ and $\varepsilon>0$ there exist $h_1\in H^q(\A)$ and
$h_2\in H^r(\A)$ such that
 $$h=h_1h_2\quad\mbox{and}\quad \|h_1\|_q\,\|h_2\|_r\le \|h\|_p+O(\varepsilon).$$
Consequently,
$$
H^p(\A)=H^q(\A)\odot H^r(\A).
$$
\end{theorem}
\begin{proof} The result is true for the case $p\geq1$, as established in \cite[Theorem 2.4]{ZJ}. Now, we consider
 the case when  $0<p_0 < 1$. Choose  $n\in\mathbb{N}$  such that $np\ge1$.  Let $h=u|h|$ be the
polar decomposition of $h$. Then
 $$
  h=v|h|^{p/nq} |h|^{p/nq}\cdots\, |h|^{p/nq}|h|^{p/nr} |h|^{p/nr}\cdots\, |h|^{p/nr}.
 $$
 From the proof of Theorem \ref{thm:reisz}, we know that 
for any  $\varepsilon>0$, there exist $x_k\in H^{nq}(\mathcal{A})$, $y_k\in H^{nr}(\mathcal{A})$
$(k=1,2,\cdots, n)$ and a contraction $a\in\A$ such that
$h=ax_1x_2\cdots x_ny_1y_2\cdots y_n$, $\|x_k\|_{nq}\le\|h\|_p^{\frac{p}{nq}}+\varepsilon$ and $\|y_k\|_{nr}\le\|h\|_p^{\frac{p}{nr}}+\varepsilon$ $(k=1,2,\cdots, n)$.
Let $h_1=ax_1x_2\cdots x_n$ and $h_2=y_1y_2\cdots y_n$. Then $h_1\in H^q(\A)$ and
$h_2\in H^r(\A)$ such that
 $$h=h_1h_2\quad\mbox{and}\quad \|h_1\|_q\,\|h_2\|_r\le \|h\|_p+O(\varepsilon).$$

\end{proof}

\begin{theorem} \label{thm:interpolationhp1} Let $\mathcal{A}$ be a type 1 subdiagonal subalgebra. If $0<\theta<1$ and $0<p_0<p_1\le\8$, then
$$
(H^{p_1}(\A)^\eta,H^{p_0}(\A))_\theta=i_{p_\theta}^{\eta}( H^{p_\theta}(\A)),
$$
 where $p_\theta=\frac{1-\theta}{p_0}+\frac{\theta}{p_1}$.
\end{theorem}
\begin{proof} We only need to prove the case $0<p_0<1$. Assume that $p_1<\8$. By Proposition \ref{pro:trueline}, we get that
$$
\|i_{p_\theta}^{\eta}(x)\|_{p_\theta}\le C\|i_{p_\theta}^{\eta}(x)\|_\theta,\qquad \forall x\in H^{p_1}(\A)^\eta.
$$

Conversely, first assume that  $\frac{1}{2}\le p_0<1$.  Let $h\in H^{p_\theta}(\mathcal{A})$. From the proof Theorem \ref{thm:reisz}, it follows that for $\varepsilon>0$,  there exists $h_k\in H^{2p_\theta}(\mathcal{A})$,
$ k=1,2$  such that
\beq\label{eq:reise-2}
h=h_1h_2, \quad\|h_k\|_{2p_\theta}\le\|h\|_{p_\theta}^{\frac{1}{2}}+\varepsilon,\quad
k=1,2.
\eeq
 Applying \cite[Theorem 3.3]{BR1} and \eqref{eq:analytic-inter}, we conclude that there exist $g_1\in A\F( H^{2p_1}(\mathcal{A})^1, H^{2p_0}(\mathcal{A}))$  and $g_2\in A\F( H^{2p_1}(\mathcal{A})^0, H^{2p_0}(\mathcal{A}))$such that $g_1(\theta)=h_1$, $g_2(\theta)=h_2$,
\beq\label{eq:interpolation-2p}
\begin{array}{l}
\|g_1\|_{A\F_a( H^{n2p_1}(\mathcal{A})^1, H^{2p_0}(\mathcal{A}))}<\|h_1\|_{2p_\theta}+\varepsilon, \\
\|g_2\|_{A\F_a( H^{2p_1}(\mathcal{A})^0, H^{2p_0}(\mathcal{A}))}<\|h_2\|_{2p_\theta}+\varepsilon. 
\end{array}
\eeq
Set $g=g_1g_2$. Then $g\in A\F( H^{p_1}(\mathcal{A})^\frac{1}{2}, H^{p_0}(\mathcal{A}))$ and 
$ g(\theta) = h$.

 Hence, by  \eqref{eq:reise-2} and \eqref{eq:interpolation-2p}, we get that
$$
\begin{array}{rl}
    \|h\|_{\theta}&\le\big\|g\big\|_{A\F( H^{2p_1}(\mathcal{A})^\frac{1}{2}, H^{2p_0}(\mathcal{A}))}\\
    &=\max\{\sup_{t\in\mathbb{R}}\|g(it)\|_{p_0},\;\sup_{t\in\mathbb{R}}\|g(1+it)\|_{p_1}\}\\
     &\le \max\{\sup_{t\in\mathbb{R}}\Pi_{k=1}^2\|g_k(it)\|_{2p_0},\sup_{t\in\mathbb{R}}\Pi_{k=1}^2\|g_k(1+it)\|_{2p_1}\}\\
    &\le \Pi_{k=1}^2\max\{\sup_{t\in\mathbb{R}}\|g_k(it)\|_{2p_0},\sup_{t\in\mathbb{R}}\|g_k(1+it)\|_{2p_1}\}\\
    &= \|g_1\|_{A\F_a( H^{n2p_1}(\mathcal{A})^1, H^{2p_0}(\mathcal{A}))}\|g_2\|_{A\F_a( H^{2p_1}(\mathcal{A})^0, H^{2p_0}(\mathcal{A}))}\\
    &< (\|h_1\|_{2p_\theta}+\varepsilon)(\|h_2\|_{2p_\theta}+\varepsilon)\\
    &<(\|h\|_{p_\theta}^{\frac{1}{2}}+2\varepsilon)^2.
  \end{array}
$$
Letting $\varepsilon\rightarrow0$, we get
$$
\|h\|_{\theta}\le \|h\|_{p_\theta}.
$$
It follows that
$$
(H^{p_1}(\A)^\frac{1}{2},H^{p_0}(\A))_\theta=i_{p_\theta}^\frac{1}{2}( H^{p_\theta}(\A)).
$$
By \eqref{eq:analytic-inter}, this implies that
$$
(H^{p_1}(\A)^\eta,H^{p_0}(\A))_\theta=i_{p_\theta}^{\eta}( H^{p_\theta}(\A)).
$$

If $\frac{1}{4} \leq p < \frac{1}{2}$, arguing as above and using the previously proven case, we can conclude as before. An easy induction procedure allows us to obtain the desired result.

Finally, we use \cite[Theorem 3.4]{BR1} and \cite[Lemma 1]{W} along with the above method to prove this result for the case when $p_1 = \infty$.
\end{proof}

Recall that the dual space of $H^1(\A)$ is the space
 $BMO(\A)$  defined in \cite{J3}  (see \cite[Theorem 3.10]{J3}). Using \cite[Theorem 3.5]{BR1},  \cite[Lemma 1]{W} and the method in the proof of Theorem \ref{thm:interpolationhp1}, we obtain the following result.

\begin{theorem} \label{thm:interpolationhp2} Let $\mathcal{A}$ be a type 1 subdiagonal subalgebra. If $0<\theta<1$ and $0<p_0<p<\8$, then
$$
(BMO(\mathcal{A})^\eta, H^{p_0}(\A))_\theta=i_{p}^{\eta}( H^{p}(\A)),
$$
 where $\theta=\frac{p_0}{p}$.
\end{theorem}

\subsection*{Acknowledgement} We thank the reviewers for  useful comments, which improved the paper. This research was funded by the Science Committee of the Ministry of Science and High
Education of the Republic of Kazakhstan (Grant No. AP14871523).

\end{document}